\documentclass{amsart}

\usepackage{amsrefs}
\usepackage[all]{xy}
\usepackage{syntonly}
\usepackage{enumerate}
\usepackage{hyperref}
\usepackage{amsfonts}
\usepackage{amssymb}
\usepackage{amsmath}
\usepackage{amsthm}
\usepackage[inline]{enumitem}
\usepackage{tikz}
\usepackage{graphics}
\usepackage{graphicx}
\usepackage{color}
\usepackage{comment}
\usepackage{caption,lipsum}

\usepackage{lineno}

\usepackage{multirow}

\newtheorem{theorem}{Theorem}[section]

\newtheorem{definition}[theorem]{Definition}
\newtheorem{lemma}[theorem]{Lemma}

\newtheorem{globalClaim}{Claim}[subsection]

\newtheorem{question}[theorem]{Question}

\newtheorem{fact}[theorem]{Fact}

\newtheorem{remark}[theorem]{Remark}

\newtheorem*{questionstar}{Question}

\specialcomment{com}{ \color{red}\Large }{ \normalsize \color{black}} 
\specialcomment{oldcom}{ \color{brown} }{\color{black}} 

\excludecomment{oldcom} 

\usepackage{stmaryrd}

\begin{document}
\title[A variant of Shelah's characterization of Strong Chang's Conjecture]{A variant of Shelah's characterization of Strong Chang's Conjecture}

\author{Sean Cox}
\email{scox9@vcu.edu}
\address{
Department of Mathematics and Applied Mathematics \\
Virginia Commonwealth University \\
1015 Floyd Avenue \\
Richmond, Virginia 23284, USA 
}

\author{Hiroshi Sakai}
\address{
Graduate School of System Informatics \\
Kobe University \\
1-1 Rokkodai \\ Nada \\ Kobe 657-8501 \\ Japan
}

\thanks{Part of this work was done during the \emph{Ideal Fest} workshop at VCU in May 2017.  The authors gratefully acknowledge support from the VCU Presidential Research Quest Fund, and Simons Foundation grant 318467.}


\begin{abstract}
Shelah~\cite{MR1623206} considered a certain version of Strong Chang's Conjecture, which we denote $\text{SCC}^{\text{cof}}$, and proved that it is equivalent to several statements, including the assertion that Namba forcing is semiproper.  We introduce an apparently weaker version, denoted $\text{SCC}^{\text{split}}$, and prove an analogous characterization of it.  In particular, $\text{SCC}^{\text{split}}$ is equivalent to the assertion that the the Friedman-Krueger poset is semiproper.  This strengthens and sharpens the results of Cox~\cite{Cox_Nonreasonable}, and sheds some light on problems from Usuba~\cite{MR3248209} and Torres-Perez and Wu~\cite{MR3431031}.
\end{abstract}

\maketitle

\section{Introduction}

Foreman-Magidor-Shelah~\cite{MR924672} considered a strong version of Chang's Conjecture,\footnote{Their version asserts that for every stationary $S \subseteq \omega_1$ and every $F:[\omega_2]^{<\omega} \to \omega_2$, there exists an $X \subset \omega_2$ such that $X$ is closed under $F$, $|X|=\omega_1$, and $X \cap \omega_1 \in S$.} which they used to show that, under Martin's Maximum, the saturation of the nonstationary ideal on $\omega_1$ cannot be destroyed by c.c.c.\ forcing.  Their version can also be used to prove stronger saturation properties of the nonstationary ideal (see the recent Dow-Tall~\cite{MR3744886}, Lemma 3.11).  Todorcevic~\cite{MR1261218} considered a strictly stronger version of Chang's Conjecture, which we denote SCC, and proved that SCC implies $\text{WRP}([\omega_2]^\omega)$, which means that every stationary subset of $[\omega_2]^\omega$ reflects to some ordinal of size $\omega_1$.\footnote{SCC is strictly stronger than the version from Foreman-Magidor-Shelah~\cite{MR924672}, since the version from the latter is preserved by adding a Cohen real, whereas SCC is not; see Todorcevic~\cite{MR1261218}. }  Shelah~\cite{MR1623206} considered an apparently stronger version, which we denote $\text{SCC}^{\text{cof}}$, and proved the following interesting characterization of it (see Section \ref{sec_Prelims} for the definition of $\text{SCC}^{\text{cof}}$ and other terms):

\begin{theorem}[Shelah]\label{thm_Shelah}
The following are equivalent:
\begin{enumerate}[label=(\alph*)]
 \item $\text{SCC}^{\text{cof}}$
 \item Namba forcing is semiproper.
 \item There exists some semiproper forcing that forces $\omega_2^V$ to be $\omega$-cofinal.
 \item\label{item_ShelahGame} Player II has a winning strategy in the following game of length $\omega$.  Player I plays $F_n: \omega_2 \to \omega_1$, player II responds by an ordinal $\delta_n < \omega_1$.  Player II wins iff, letting $\delta_\omega:= \text{sup}_n \delta_n$, there are cofinally many $\alpha < \omega_2$ such that $\forall n \in \omega \ F_n(\alpha) < \delta_\omega$.  
 
 We will denote this game $\boldsymbol{\mathcal{G}}^{\textbf{cof}}$.
 \item  For every Skolemized structure $\mathfrak{A}$ in a countable language extending \[(H_{\omega_3},\in), \]   the particular strategy, where II plays $\omega_1 \cap \text{Hull}^{\mathfrak{A}}(F_0,F_1, \dots, F_n)$ in the game $\mathcal{G}^{\text{cof}}$ described in part \ref{item_ShelahGame},\footnote{$\text{Hull}^{\mathfrak{A}}(X)$ denotes the Skolem hull of $X$ in the structure $\mathfrak{A}$. } is a winning strategy for player II in that game. 
\end{enumerate}
\end{theorem}

In \cite{MR2276627}, Friedman and Krueger considered the poset that adds a Cohen real and then shoots a club to kill the stationarity of $([\omega_2]^\omega)^V$ using countable conditions; let $\mathbb{Q}_{\text{FK}}$ denote this poset, which always preserves stationary subsets of $\omega_1$.  We will sometimes refer to it as the \emph{Friedman-Krueger poset}.  They asked whether ZFC proves that $\mathbb{Q}_{\text{FK}}$ is semiproper.  This was answered negatively in Cox~\cite{Cox_Nonreasonable}, where it was shown that semiproperness of $\mathbb{Q}_{\text{FK}}$ implies SCC, and hence has large cardinal consistency strength.  In fact, in \cite{Cox_Nonreasonable}, semiproperness of $\mathbb{Q}_{\text{FK}}$ was sandwiched between two versions of Strong Chang's Conjecture, though those two versions were also shown there to be non-equivalent.\footnote{Specifically, the principle $\text{SCC}^{\text{cof}}_{\text{gap}}$ was shown to imply semiproperness of $\mathbb{Q}_{\text{FK}}$, which in turn was shown to imply SCC.  That $\text{SCC}^{\text{cof}}_{\text{gap}}$ is strictly stronger than SCC, and even strictly stronger than $\text{SCC}^{\text{cof}}$, was shown in Section 3 of \cite{Cox_Nonreasonable}. }  

In this paper we introduce another version of Strong Chang's Conjecture, denoted $\text{SCC}^{\text{split}}$, and prove Theorem \ref{thm_CoxSakai} below, which is analogous to Shelah's Theorem \ref{thm_Shelah}.  In particular, Theorem \ref{thm_CoxSakai} exactly characterizes semiproperness of $\mathbb{Q}_{\text{FK}}$, and tightens the results from \cite{Cox_Nonreasonable}:

\begin{theorem}[Main Theorem]\label{thm_CoxSakai}
The following are equivalent:
\begin{enumerate}[label=(\alph*)]
 \item\label{item_SCCsplit} $\text{SCC}^{\text{split}}$
 \item\label{item_GitikPoset} The Friedman-Krueger poset is semiproper.
 \item\label{item_SomeGitikLike} There exists some semiproper forcing that kills the stationarity of $\big( [\omega_2]^\omega \big)^V$.
 \item\label{item_CoxGame} For every Skolemized structure $\mathfrak{A}$ in a countable language extending \[(H_{\omega_3},\in), \] Player II has a winning strategy in the following game of length $\omega$, which we denote by $\boldsymbol{\mathcal{G}}^{\textbf{split}}_{\boldsymbol{\mathfrak{A}}}$.   Player I plays $F_n:\omega_2 \to \omega_1$, and player II responds with some $\delta_n < \omega_1$.  Player II wins iff, letting $\delta_\omega:= \text{sup}_n \delta_n$, there exist $\alpha, \beta < \omega_2$ such that:
 \begin{itemize}
  \item $\forall n < \omega$, $F_n(\alpha)$ and $F_n(\beta)$ are both $<\delta_\omega$; and
  \item $\exists h: \omega_2 \times \omega_2 \to \omega_1$ such that $h \in \text{Hull}^{\mathfrak{A}}\big( \{ F_n \ : \ n \in \omega \} \big)$ and $h(\alpha,\beta) \ge \delta_\omega$.
\end{itemize}  
  \item\label{item_ParticularStrategy}  For every Skolemized structure $\mathfrak{A}$ in a countable language extending \[(H_{\omega_3},\in), \] the particular strategy, where II plays $\omega_1 \cap \text{Hull}^{\mathfrak{A}}(F_0,F_1, \dots, F_n)$ in the game $\mathcal{G}^{\text{split}}_{\mathfrak{A}}$ described in part \ref{item_CoxGame}, is a winning strategy for player II in that game. 
\end{enumerate}
\end{theorem}

Section \ref{sec_Prelims} provides some background, and Section \ref{sec_ProofCoxSakai} proves the main Theorem \ref{thm_CoxSakai}.  Section  \ref{sec_ConcludingRemarks} discusses how Theorem \ref{thm_CoxSakai} sheds light on a question that was asked directly by Usuba, but is closely related to other questions in the literature.

\section{Preliminaries, and versions of Strong Chang's Conjecture}\label{sec_Prelims}

Given sets $M$ and $N$, $M \sqsubseteq N$ means that $M \subseteq N$ and $M \cap \omega_1 = N \cap \omega_1$.  Given a poset $\mathbb{P}$, a countable $N \prec (H_\theta,\in,\mathbb{P})$, and a condition $p$, we say that $p$ is an $(M,\mathbb{P})$-semimaster condition iff for every $\dot{\alpha} \in M$ that names a countable ordinal, $p \Vdash \dot{\alpha} \in \check{M} \cap \omega_1$.  This is equivalent to requiring that $p \Vdash \check{M} \sqsubseteq \check{M}[\dot{G}]$.  We say $\mathbb{P}$ is semiproper iff for every $\theta$ with $\mathbb{P} \in H_\theta$, every countable $M \prec (H_\theta,\in,\mathbb{P})$, and every $p_0 \in M \cap \mathbb{P}$, there exists a $p \le p_0$ that is an $(M,\mathbb{P})$-semimaster condition.

We frequently use the following fact (see e.g.\ Larson-Shelah~\cite{MR2030084}):
\begin{fact}\label{fact_Hulls}
If $\theta$ is regular uncountable, $\mathfrak{A}$ is a structure on $H_\theta$ in a countable language which has definable Skolem functions, $M \prec \mathfrak{A}$, and $Y$ is a subset of some $\eta \in M$, then 
\[
\text{Hull}^{\mathfrak{A}}(M \cup Y) = \{ f(y) \ | \  y \in [Y]^{<\omega} \text{ and }  f \in M \cap {}^{[\eta]^{<\omega}} H_\theta \}.
\]
\end{fact}

\begin{definition}\label{def_SCC_variants}
We define three principles, denoted $\textbf{SCC}^{\textbf{cof}}$, $\textbf{SCC}^{\textbf{split}}$, and $\textbf{SCC}$, in parallel:  for all sufficiently large regular $\theta$ and all wellorders $\Delta$ on $H_\theta$, if $M \prec (H_\theta,\in,\Delta)$ and $M$ is countable then  
\noindent \begin{itemize}
 \item ($\textbf{SCC}^{\textbf{cof}}$:) \  $\forall \beta < \omega_2$  $\exists M' \ \ M \sqsubseteq M' \prec (H_\theta,\in,\Delta) \text{ and } \beta \le \text{sup}(M' \cap \omega_2)$.
 \item  ($\textbf{SCC}^{\textbf{split}}$:) \    $\exists M_0,M_1$ such that $M \sqsubseteq M_i \prec (H_\theta,\in,\Delta)$ for $i \in \{ 0,1\}$, and $M_0 \cap \omega_2$ is $\subseteq$-incomparable with $M_1 \cap \omega_2$.\footnote{I.e.\ $M_0 \cap \omega_2 \nsubseteq M_1 \cap \omega_2$ and $M_1 \cap \omega_2 \nsubseteq M_0 \cap \omega_2$.}
 \item (\textbf{SCC:}) \   $\exists M'$ \ $M \sqsubseteq M' \prec (H_\theta,\in,\Delta)$ and $M' \cap \omega_2 \supsetneq M \cap \omega_2$.
\end{itemize}
\end{definition}

The principles $\text{SCC}$ and $\text{SCC}^{\text{cof}}$ have been considered many times in the literature, though the terminology is highly inconsistent; see Table 1 (p.\ 622) of \cite{Cox_Nonreasonable} for a summary of their use in the literature.  The principle $\text{SCC}^{\text{split}}$ has not, as far as the authors are aware, been considered before.

\begin{remark}\label{rem_EndExtension}
If $M \sqsubseteq N$ and both are elementary in $(H_\theta,\in,\Delta)$, then $M \cap \omega_2$ is an initial segment of $N \cap \omega_2$.  This is because if $\alpha \in M \cap \omega_2$ and $f$ is the $\Delta$-least surjection from $\omega_1 \to \alpha$, then $f \in M \sqsubseteq N$ and so
\[
M \cap \alpha = f[M \cap \omega_1] = f[N \cap \omega_1] = N \cap \alpha.
\]
Hence in the definition of SCC, we could have equivalently required that $\text{sup}(M \cap \omega_2) < \text{sup}(N \cap \omega_2)$.
\end{remark}

\begin{lemma} 
\[
\text{SCC}^{\text{cof}} \ \implies \ \text{SCC}^{\text{split}} \ \ \implies \text{SCC}
\]
\end{lemma}
\begin{proof}
The right implication is obvious, since if $M_0$ and $M_1$ are both $\sqsubseteq$-extensions of $M$ whose intersections with $\omega_2$ are $\subseteq$-incomparable, then both must properly extend $M$ below $\omega_2$.  

For the left implication, consider any countable $M \prec (H_\theta,\in,\Delta)$ (with $\theta$ and $\Delta$ as in Definition \ref{def_SCC_variants}), and construct a $\sqsubseteq$-ascending chain $\langle M_i \ : \ i < \omega_1 \rangle$ of countable elementary substructures of $(H_\theta,\in,\Delta)$, such that $M_0 = M$ and $\text{sup}(M_{i+1} \cap \omega_2) > \text{sup}(M_i \cap \omega_2)$ for all $i < \omega_1$; this can be done by applying $\text{SCC}^{\text{cof}}$ (or just SCC) at successor steps, and taking unions at limit steps.  Let $X:=\bigcup_{i < \omega_1} M_i$ and $\eta:= \text{sup}(X \cap \omega_2)$.  By $\text{SCC}^{\text{cof}}$ there is some countable $N$ such that $M \sqsubseteq N \prec (H_\theta,\in,\Delta)$, and $\text{sup}(N \cap \omega_2) > \eta$.  Since $X \cap \omega_2$ is uncountable, there is some $i_0 < \omega_1$ such that $M_{i_0} \cap \omega_2 \nsubseteq N$.  But also $N \cap \omega_2 \nsubseteq M_{i_0}$ because $M_{i_0} \cap \omega_2 \subset \eta$.
\end{proof}

The following lemma is convenient for a couple of reasons.  First, it implies that if $\text{SCC}^{\text{split}}$ fails, then it fails for stationarily many $M \in [H_\theta]^\omega$.  Also, it allows us to replace ``every" with ``club-many" in the definition of $\text{SCC}^{\text{split}}$, but \emph{without} having to expand the structure in which we require elementarity of the end-extensions.  This latter feature is useful, for example, in the proof in Section \ref{sec_FromWS}.
\begin{lemma}\label{lem_FailsOnStatSet}
The following are equivalent.
\begin{enumerate}
 \item $\text{SCC}^{\text{split}}$;
 \item There are club-many $M \in [H_{\omega_3}]^\omega$ such that for $i \in \{ 0,1 \}$, there exist countable $M_i$ such that $M \sqsubseteq M_i \prec (H_{\omega_3},\in)$, and $M_0 \cap \omega_2$ is $\subseteq$-incomparable with $M_1 \cap \omega_2$.
\end{enumerate}
\end{lemma}
The proof of Lemma \ref{lem_FailsOnStatSet} is almost identical to the proof of Lemma 13 of \cite{Cox_Nonreasonable}, so we omit it.


We now state a convenient characterization of SCC, which in particular shows why it is a strong form of Chang's Conjecture.  By a \textbf{Chang set} or \textbf{Chang structure} we mean a set $X$ such that $|X \cap \omega_2|=\omega_1$ and $X \cap \omega_1 \in \omega_1$.  If $X$ is a Chang elementary substructure of $(H_\theta,\in)$, then $X \cap \omega_2$ always has ordertype exactly $\omega_1$.\footnote{This is because if $\alpha \in X \cap [\omega_1,\omega_2)$, then by elementarity there is an $f \in X$ such that $f:\omega_1 \to \alpha$ is a bijection.  It follows again by elementarity that $X \cap \alpha = f[X \cap \omega_1]$, which is countable. }
\begin{lemma}\label{lem_SCC_Equiv}
The following are equivalent:
\begin{enumerate}
 \item SCC
 \item For all sufficiently large regular $\theta$, all wellorders $\Delta$ on $H_\theta$, and all countable $M \prec (H_\theta,\in,\Delta)$, there exists a Chang set $X$ such that $X \prec (H_\theta,\in,\Delta)$ and $M \sqsubset X$. 
\end{enumerate}
\end{lemma}

If ``all countable $M$" in the second clause of Lemma \ref{lem_SCC_Equiv} is replaced by ``stationarily many countable $M$", the result is a characterization of the classic Chang's Conjecture.

\section{Proof of Theorem \ref{thm_CoxSakai}}\label{sec_ProofCoxSakai}

In this section we prove Theorem \ref{thm_CoxSakai}.  We believe it is conceptually simpler to essentially separate items \ref{item_SCCsplit} through \ref{item_SomeGitikLike} from items \ref{item_CoxGame} and \ref{item_ParticularStrategy}.  This results in one redundant step.  Specifically, we prove that
\[
\ref{item_SCCsplit} \implies \ref{item_GitikPoset} \implies \ref{item_SomeGitikLike} \implies \ref{item_SCCsplit},
\]
and then deal with the characterizations involving games, specifically
\[
\ref{item_SCCsplit} \implies \ref{item_ParticularStrategy} \implies \ref{item_CoxGame} \implies \ref{item_SCCsplit}.
\]

\subsection{\ref{item_SCCsplit} $ \implies $ \ref{item_GitikPoset}}\label{sec_HiroshiDirection}

Assume $\text{SCC}^{\text{split}}$, and fix a sufficiently large regular $\theta$ and a wellorder $\Delta$ on $H_\theta$.  Fix any countable $M \prec (H_\theta,\in,\Delta)$, we recursively define a binary tree $T_M$ of height $\omega$ isomorphic to the tree $2^{<\omega}$ as follows.  Set $M_{\langle \rangle}:=M$, and given $s \in {}^{<\omega} 2$, use $\text{SCC}^{\text{split}}$ to find countable $M_{s^\frown 0}$ and $M_{s^\frown 1}$ that both $\sqsubseteq$-extend $M_s$, are elementary in $(H_\theta,\in,\Delta)$, and are $\subseteq$-incomparable below $\omega_2$.

\begin{globalClaim}\label{clm_BranchNotInV}
If $W$ is an outer model of $V$ and $\sigma \in {}^\omega 2 \cap W \setminus V$, then 
\[
\omega_2^V \cap \bigcup_{n \in \omega} M_{\sigma \restriction n} \notin V.
\]
\end{globalClaim}
\begin{proof}
(of Claim \ref{clm_BranchNotInV}):  Suppose toward a contradiction that $z:=\omega_2^V \cap \bigcup_{n \in \omega} M_{\sigma \restriction n}$ is an element of $V$; we will use $z$ and $T_M$ to define $\sigma$ in $V$, yielding a contradiction.  In $W$, set 
\[
b_\sigma:= \langle M_{\sigma \restriction n} \ : \ n < \omega \rangle.
\]

To decode $\sigma$ from $z$ and $T_M$, consider first the root $M = M_{\langle \rangle}$ of $T_M$, and its two immediate successors, $M_{\langle 0 \rangle}$ and $M_{\langle 1 \rangle}$.  Now:
\begin{itemize}
 \item $M_{\langle 0 \rangle} \cap \omega_2 \setminus M_{\langle 1 \rangle} \cap \omega_2$ and  $M_{\langle 1 \rangle} \cap \omega_2 \setminus M_{\langle 0 \rangle} \cap \omega_2$ are both nonempty;
 \item Models in $T_M$ along a given branch end-extend each other below $\omega_2^V$; and
 \item $z$ is the intersection of $\omega_2^V$ with the (union of) the branch $b_\sigma$.
\end{itemize}
It follows that $z$ contains, as an initial segment, \emph{exactly one} element of the set $\{M_{\langle 0 \rangle} \cap \omega^V_2, M_{\langle 1 \rangle} \cap \omega^V_2 \}$, and moreover that the one it contains corresponds to the model at level 1 of the branch $b_\sigma$.  This tells us which node at level 1 of $T_M$ is in $b_\sigma$.  Continuing in this manner allows one to reconstruct $b_\sigma$, and hence $\sigma$, from $z$ and $T_M$.  Hence $\sigma \in V$, a contradiction.\footnote{More precisely, using the contradiction assumption that $z \in V$, working $V$ we can recursively define the function $r:\omega \to 2$ by letting $r(n)$ be the unique node immediately above $M_{r \restriction n}$ in $T_M$ whose intersection with $\omega_2$ is an initial segment of $z$, if such a node exists.  Then in $W$, it is routine to recursively check that $z$ contains, as an initial segment, exactly one element of the set $\{ M_{(r \restriction n)^\frown 0} \cap \omega_2^V, M_{(r \restriction n)^\frown 1} \cap \omega_2^V\}$, and moreover that the one it contains is the model at level $n+1$ of $b_\sigma$.  Hence $r= \sigma$, and so $\sigma \in V$, a contradiction.}
\end{proof}

Now suppose $(s,\dot{f})$ is a condition in $M \cap \mathbb{Q}_{\text{FK}}$; we need to find an $(M,\mathbb{Q}_{\text{FK}})$-semimaster condition below $(s,\dot{f})$.  As proved at the beginning of Section 5.1 of \cite{Cox_Nonreasonable}, it suffices to show that if $\sigma$ is $\text{Add}(\omega)$-generic with $s \in \sigma$, then $V[\sigma] \models$ ``there exists an $N \prec (H_\theta[\sigma],\in)$ such that $M[\sigma] \sqsubseteq N$ and $N \cap \omega_2^V \notin V$".  We claim that $N:=M_\omega[\sigma]$ works, where 
\[
M_\omega:= \bigcup_{n < \omega} M_{\sigma \restriction n}.
\]

Clearly $M[\sigma] \subseteq M_\omega[\sigma]$, since $M$ is the root of the tree $T_M$.  Also, it is a standard fact that for every $n < \omega$, $M_{\sigma \restriction n}[\sigma]$ is an elementary substructure of $(H_\theta[\sigma],\in)$.  Now
\begin{equation}\label{eq_GenericUnion}
M_\omega[\sigma] = \bigcup_{n < \omega} M_{\sigma \restriction n} [\sigma]
\end{equation}
and it follows that $M_\omega[\sigma]$ is also elementary in $(H_\theta[\sigma], \in)$.  By Claim \ref{clm_BranchNotInV}, $\omega^V_2  \cap M_\omega \notin V$.  Also, because every model in $T_M$ is a $\sqsubseteq$-extension of $M$, 
\[
V[\sigma] \models \ M \cap \omega_1 = M_\omega \cap \omega_1.
\]
The following claim completes the proof:
\begin{globalClaim}\label{clm_SigmaAddsNoORDS}
$V[\sigma] \models M_\omega \cap \omega_1 = M_\omega[\sigma] \cap \omega_1$.  (In fact we prove they have the same intersection with $\text{ORD}$).
\end{globalClaim}
\begin{proof}
(of Claim \ref{clm_SigmaAddsNoORDS}):  Since $\sigma$ is generic for a c.c.c.\ forcing, $\sigma$ includes a master condition, namely $\emptyset$, for every countable elementary model from $V$.  In particular, $\sigma$ includes a master condition for $M_{\sigma \restriction n}$ for every $n < \omega$, so
\[
\forall n < \omega \ M_{\sigma \restriction n} \cap \text{ORD} = M_{\sigma \restriction n}[\sigma] \cap \text{ORD}.
\]
Together with \eqref{eq_GenericUnion}, this completes the proof of the claim.
\end{proof}

\subsection{\ref{item_GitikPoset} $\implies$ \ref{item_SomeGitikLike}}

This direction is trivial, since the Friedman-Krueger poset clearly kills the stationarity of $([\omega_2]^\omega)^V$.

\subsection{\ref{item_SomeGitikLike} $\implies$ \ref{item_SCCsplit}}

Assume $\mathbb{P}$ is a semiproper poset that kills stationarity of $([\omega_2]^\omega)^V$.  First we claim:
\begin{globalClaim}\label{clm_SCC}
SCC holds.
\end{globalClaim}
\begin{proof}
(of Claim \ref{clm_SCC}):  this is just Theorem 22 of Cox~\cite{Cox_Nonreasonable}, but we provide a brief sketch (please refer to \cite{Cox_Nonreasonable} for more details).  Let $M \prec (H_\theta,\in,\Delta)$ be countable.  Since $([\omega_2]^\omega)^V$ is nonstationary in $V^{\mathbb{P}}$, and assuming WLOG that $M$ includes a name for a function from $[\omega_2^V]^{<\omega} \to \omega_2^V$ witnessing this fact, it follows that $M[G] \cap \omega_2^V \notin V$ whenever $G$ is $\mathbb{P}$-generic; in particular $M[G] \cap \omega_2^V \ne M \cap \omega_2^V$.  Pick any $\beta$ in the difference; then
\[
M \subsetneq \text{Hull}^{(H^V_\theta, \in,\Delta)}(M \cup \{ \beta \}) \subseteq \ M[G]
\]
Since $\mathbb{P}$ is semiproper then we can choose $G$ so that it includes an $(M,\mathbb{P})$-semimaster condition, so that $M \sqsubseteq M[G]$.  It follows that, in $V$, $\text{Hull}^{(H^V_\theta, \in,\Delta)}(M \cup \{ \beta \})$ is a $\sqsubset$-extension of $M$ that includes the ``new" ordinal $\beta$. 
\end{proof}

So \ref{item_SomeGitikLike} implies SCC, but we want $\text{SCC}^{\text{split}}$.  Now assume toward a contradiction that $\text{SCC}^{\text{split}}$ fails.  By Lemma \ref{lem_FailsOnStatSet}, it fails for stationarily many elements of $[H_\theta]^\omega$; let $S$ denote this stationary set.  For each $M \in S$, use Claim \ref{clm_SCC} and Lemma \ref{lem_SCC_Equiv} to choose a Chang set $X_M$ such that $M \sqsubset X_M \prec (H_\theta,\in,\Delta)$.  Below, ``Hull" refers to the Skolem hull in the structure $(H_\theta,\in,\Delta)$.

\begin{globalClaim}\label{clm_UniqueChang}
Suppose $M \in S$.  Then whenever $Y$ is a Chang $\sqsubseteq$-extension of $M$ that is elementary in $(H_\theta,\in,\Delta)$, $X_M \cap \omega_2 = Y \cap \omega_2$.  
\end{globalClaim}
\begin{proof}
(of Claim \ref{clm_UniqueChang}):  we prove that $X_M \cap \omega_2 \subseteq Y \cap \omega_2$; the other direction is similar.  Suppose toward a contradiction that there is some $\beta \in (X_M \cap \omega_2) \setminus (Y \cap \omega_2)$.  Then $M(\beta):=\text{Hull}(M \cup \{ \beta \}) \subseteq X_M$ and is a countable $\sqsubseteq$ extension of $M$.  In particular, since $Y \cap \omega_2$ is uncountable, $Y \nsubseteq M(\beta)$.  Fix any $\eta \in Y \setminus M(\beta)$.  Then $M(\eta):=\text{Hull}(M \cup \{ \eta \}) \subset Y$ and is a countable $\sqsubseteq$ extension of $M$.  Finally, note that $\beta \notin M(\eta)$ and $\eta \notin M(\beta)$; so $M(\beta) \cap \omega_2$ and $M(\eta) \cap \omega_2$ are $\subseteq$-incomparable.  This contradicts that $M \in S$.
\end{proof}

\begin{globalClaim}\label{clm_ExtOfMContainedXM}
If $M \in S$, then whenever $Q \prec (H_\theta,\in,\Delta)$ is countable and $Q \sqsupseteq M$, $Q \cap \omega_2$ is an initial segment of $X_M \cap \omega_2$.
\end{globalClaim}
\begin{proof}
(of Claim \ref{clm_ExtOfMContainedXM}): By Remark \ref{rem_EndExtension} it suffices to prove that $Q  \cap \omega_2 \subseteq X_M \cap \omega_2$.  Suppose toward a contradiction that $Q \cap \omega_2 \nsubseteq X_M \cap \omega_2$; fix some $\beta \in (Q \cap \omega_2) \setminus (X_M \cap \omega_2)$.  Claim \ref{clm_SCC} and Lemma \ref{lem_SCC_Equiv} ensure that there is a Chang $\sqsubseteq$-extension $Y$ of $Q$ that is elementary in $(H_\theta,\in,\Delta)$.  Since $M \sqsubseteq Q$, $Y$ is also a Chang extension of $M$.  But $\beta \in (Y \cap \omega_2) \setminus ( X_M \cap \omega_2)$, which contradicts Claim \ref{clm_UniqueChang}.
\end{proof}

Let $\dot{F}:[\omega_2^V]^{<\omega} \to \omega_2^V$ be a $\mathbb{P}$-name witnessing that $\mathbb{P}$ kills the stationarity of $([\omega_2]^\omega)^V$; so $\mathbb{P}$ forces that every countable set closed under $\dot{F}$ fails to be in the ground model.

Fix an $M \in S$ such that $\dot{F} \in M$.  Then
\[
\Vdash \ \check{M}[\dot{G}] \cap \omega_2^V \notin V.
\]
Since $\mathbb{P}$ is semiproper, there exists some $(M,\mathbb{P})$-semimaster condition $p$.  Let $G$ be generic with $p \in G$.  Then $M \sqsubseteq M[G]$ and $M[G] \cap \omega_2^V \notin V$.  In $V[G]$, let $N:=M[G]$.  Let $\eta$ be the \emph{least} ordinal $\le \omega_2^V$ such that $N \cap \eta \notin V$.\footnote{$\eta$ might equal $\omega_2^V$, e.g.\ if $\mathbb{P}$ is Namba forcing.}   

Consider any $\zeta < \eta$.  Then $N \cap \zeta \in V$, and hence $\text{Hull}(N \cap \zeta) \in V$ (here the hull is in $(H_\theta,\in,\Delta)$).  Then by Claim \ref{clm_ExtOfMContainedXM}, $\omega_2 \cap \text{Hull}(N \cap \zeta)$ is an initial segment of $X_M \cap \omega_2$.  It follows that
\[
N \cap \eta = \bigcup_{\zeta < \eta} \text{Hull}(N \cap \zeta) \text{ is an initial segment of } X_M \cap \omega_2.
\]
But $X_M \cap \omega_2 \in V$ and hence so are all of its initial segments.  So $N \cap \eta \in V$, a contradiction.

\subsection{\ref{item_SCCsplit} $\implies$ \ref{item_ParticularStrategy}}

Assume $\text{SCC}^{\text{split}}$, and at stage $n$ of the game $\mathcal{G}^{\text{split}}_{\mathfrak{A}}$, have player II play
\[
\omega_1 \cap \text{Hull}^{\mathfrak{A}}(F_0,F_1,\dots,F_n) .
\]
We claim this is a winning strategy for player II in $\mathcal{G}^{\text{split}}_{\mathfrak{A}}$.  Let 
\[
X:= \text{Hull}^{\mathfrak{A}}(\{ F_n \ : \ n \in \omega \})
\]
By $\text{SCC}^{\text{split}}$ there exist $\alpha,\beta < \omega_2$ such that $X(\alpha):=\text{Hull}^{\mathfrak{A}}(X \cup \{ \alpha \})$ and $X(\beta):=\text{Hull}^{\mathfrak{A}}(X \cup \{ \beta \})$ both $\sqsupset$-extend $X$, but are $\subseteq$-incomparable below $\omega_2$. 

\begin{globalClaim}\label{clm_Someh}
Let $X(\alpha,\beta):=\text{Hull}^{\mathfrak{A}}(X \cup \{ \alpha,\beta\})$.  Then $X(\alpha,\beta)  \cap \omega_1 > X \cap \omega_1$.
\end{globalClaim}
\begin{proof}
(of Claim \ref{clm_Someh}):  suppose toward a contradiction that they are equal.  Then
\[
X \sqsubseteq X(\alpha) \sqsubseteq X(\alpha,\beta) \text{ and } X \sqsubseteq X(\beta) \sqsubseteq X(\alpha,\beta).
\]
Then by Remark \ref{rem_EndExtension}, $X(\alpha,\beta) \cap \omega_2$ end-extends both $X(\alpha) \cap \omega_2$ and $X(\beta) \cap \omega_2$.  But this implies that one of $X(\alpha) \cap \omega_2$ and $X(\beta) \cap \omega_2$ is a subset of the other, contrary to our choice of $\alpha$ and $\beta$.
\end{proof}

By Claim \ref{clm_Someh} and Fact \ref{fact_Hulls}, there is some $h: \omega_2 \times \omega_2 \to \omega_1$ with $h \in X$ such that $h(\alpha, \beta) \ge X \cap \omega_1$.  But note by definition of $X$ that $h \in \text{Hull}^{\mathfrak{A}}(\{ F_n \ : \ n \in \omega \})$.  This takes care of the final requirement in the definition of ``II wins" in the game $\mathcal{G}^{\text{split}}_{\mathfrak{A}}$.  Note that since $F_n \in X$ for every $n$, and by the fact that $X \sqsubset X(\alpha)$ and $X \sqsubset X(\beta)$, we have the other requirements satisfied as well.  So player II wins the game.

\subsection{\ref{item_ParticularStrategy} $\implies$  \ref{item_CoxGame}}  This direction is trivial.

\subsection{\ref{item_CoxGame} $\implies$ \ref{item_SCCsplit}}\label{sec_FromWS}

Let $\Delta$ be a wellorder on $H_{\omega_3}$, let $\mathfrak{A}=(H_{\omega_3},\in,\Delta)$, and suppose Player II has a winning strategy in the game $\mathcal{G}^{\text{split}}_{\mathfrak{A}}$.  We want to prove that $\text{SCC}^{\text{split}}$ holds.  By Lemma \ref{lem_FailsOnStatSet}, it suffices to show that for club-many $M \in [H_{\omega_3}]^\omega$, there exist $\alpha,\beta < \omega_2$ such that 
\[
\omega_1 \cap \text{Hull}^{\mathfrak{A}}(M \cup \{ \alpha \}) = \omega_1 \cap \text{Hull}^{\mathfrak{A}}(M \cup \{ \beta \}) = \omega_1 \cap M
\]
but $\omega_1 \cap \text{Hull}^{\mathfrak{A}}(M \cup \{ \alpha,\beta\}) > \omega_1 \cap M$.  Note that this would imply in particular that $\alpha \notin \text{Hull}^{\mathfrak{A}}(M \cup \{ \beta \})$ and $\beta \notin \text{Hull}^{\mathfrak{A}}(M \cup \{ \alpha \})$.

We claim this is true for every countable $M$ that is elementary in the expanded structure $\mathfrak{A}^\frown \sigma$, where $\sigma$ is any winning strategy for Player II.\footnote{Note that $\sigma$ can be viewed as a predicate on $H_{\omega_3}$, which is the universe of $\mathfrak{A}$, so the expanded structure makes sense.}  Fix such an $M$, and let $\langle F_n \ : \ n \in \omega \rangle$ enumerate $M \cap {}^{\omega_2} \omega_1$.   Define a run of the game, where Player I plays the $F_n$'s, and Player II responds according to the strategy $\sigma$.  Let $\alpha,\beta$, and $h \in {}^{\omega_2 \times \omega_2} \omega_1 \cap  \text{Hull}^{\mathfrak{A}}(\{ F_n \ : \ n \in \omega \})$ be witnesses to the fact that II wins the game (as defined in clause \ref{item_CoxGame} of Theorem \ref{thm_CoxSakai}).

Note that since $M \prec \mathfrak{A}^\frown \sigma$ and Player I's moves are functions from $M$, the output of $\sigma$ at each stage $n$ of the game, which we will denote by $\delta_n$, is an element of $M \cap \omega_1$.  Hence
\[
M \cap \omega_1 \ge \delta_\omega=\text{sup}_n \delta_n.
\]
On the other hand, for every $\xi \in M \cap \omega_1$, there is some $n < \omega$ such that $F_{n}$ has constant value $\xi$, and hence, since $\alpha,\beta$ witness that Player II wins, in particular $\xi = F_n(\alpha) < \delta_\omega$.  So in fact
\begin{equation}\label{eq_M_not_extended}
M \cap \omega_1 = \delta_\omega= \text{sup}_n \delta_n
\end{equation}

Now since $\alpha,\beta,h$ witness that Player II wins the game, and since $\vec{F}$ enumerates exactly $M \cap {}^{\omega_2} \omega_1$, Fact \ref{fact_Hulls} ensures that $M$, $M(\alpha):=\text{Hull}^{\mathfrak{A}}(M \cup \{ \alpha \})$, and $ M(\beta):=\text{Hull}^{\mathfrak{A}}(M \cup \{ \beta \})$ all have the same intersection with $\omega_1$, namely $\delta_\omega$.  Here we emphasize that the hulls are taken in $\mathfrak{A}$, \textbf{not} in the expanded structure $\mathfrak{A}^\frown \sigma$.  On the other hand, $h \in \text{Hull}^{\mathfrak{A}}(\{ F_n \ : \ n \in \omega \}) \subseteq M$, and hence 
\begin{equation}\label{eq_h_in_M}
h \in M.
\end{equation}
Since $\alpha,\beta,h$ witness that Player II wins the game, $h(\alpha,\beta) \ge \delta_\omega$.  But since $h \in M$, this implies that $h(\alpha,\beta) \in \text{Hull}^{\mathfrak{A}}(M \cup \{ \alpha,\beta \})$, and hence the latter's intersection with $\omega_1$ is strictly larger than $M \cap \omega_1$.

\section{Concluding Remarks}\label{sec_ConcludingRemarks}

Consider the following implications discussed earlier (the last implication is the one due to Todorcevic~\cite{MR1261218}, mentioned in the introduction):
\begin{equation}\label{eq_ChainImp}
\text{SCC}^{\text{cof}} \implies \text{SCC}^{\text{split}} \implies \text{SCC} \implies \text{WRP}([\omega_2]^\omega).
\end{equation}
Usuba asked:
\begin{questionstar}[Usuba~\cite{MR3248209}, Question 3.14 part 4]\label{q_Usuba}
Is $\text{SCC}^{\text{cof}}$ equivalent to SCC?
\end{questionstar}
In light of our Theorem \ref{thm_CoxSakai}, Shelah's Theorem \ref{thm_Shelah}, and the implications in \eqref{eq_ChainImp}, a positive answer to Usuba's Question would imply that semiproperness of the Friedman-Krueger poset implies semiproperness of Namba forcing.  We conjecture this is false.  

On a related topic, Torres-Perez and Wu proved in \cite{MR3431031} that $\text{SCC}^{\text{cof}}$, together with failure of $\text{CH}$, implies the Tree Property at $\omega_2$.  They asked (Question 4.1 of \cite{MR3431031}) whether their assumption of $\text{SCC}^{\text{cof}}$ could be weakened to ($\neg \text{CH}$ plus) $\text{WRP}([\omega_2]^\omega)$.  In light of the implications in \eqref{eq_ChainImp}, it is also natural to ask if their assumption could be weakened to ($\neg \text{CH}$ plus) either $\text{SCC}^{\text{split}}$ or $\text{SCC}$.  We conjecture that their assumptions cannot be significantly weakened; i.e.\ that $\neg \text{CH}$ plus $\text{SCC}$ (and possibly even $\text{SCC}^{\text{split}}$) is consistent with an $\omega_2$-Aronszajn tree.

We also include a technical question.  The proof that $\text{SCC}^{\text{split}}$ implies semiproperness of the Friedman-Krueger poset made use of the fact that the first step of the Friedman-Krueger poset (i.e.\ Cohen forcing) is c.c.c.; this was used in the proof of Claim \ref{clm_SigmaAddsNoORDS}, to ensure that the generic real includes a master condition for every model along the generic branch $b_\sigma$ of the tree $T_M$.  More generally, the proof of Theorem \ref{thm_CoxSakai} shows that $\text{SCC}^{\text{split}}$ is equivalent to semiproperness for any poset of the form ``add a new real and then shoot a club through $[\omega_2^V]^\omega \setminus V$", \emph{provided that} the forcing to add the new real $\sigma$ can be arranged to include a semimaster condition for \emph{every} model along the branch $b_\sigma$ from the proof.  This raises the following:
\begin{question}
For which forcings of the form ``add a new real and then shoot a club through $[\omega_2^V]^\omega \setminus V$" is semiproperness equivalent to $\text{SCC}^{\text{split}}$?  

Note that by the equivalence of clauses \ref{item_SCCsplit} and \ref{item_SomeGitikLike} of Theorem \ref{thm_CoxSakai}, semiproperness of any such poset implies $\text{SCC}^{\text{split}}$.  But it's not clear if, for example, there exists a proper forcing $\mathbb{P}$ adding a new real, such that semiproperness of ``$\mathbb{P}$ followed by shooting a club through $[\omega_2^V]^\omega \setminus V$" is \textbf{strictly} stronger than $\text{SCC}^{\text{split}}$.
\end{question}

\end{document}